\documentclass[a4paper,UKenglish,cleveref, autoref,numberwithinsect]{socg-lipics-v2021}

\usepackage{graphicx}
\usepackage{color}
\usepackage{amsmath}
\usepackage{amssymb}
\usepackage{xspace}
\usepackage{epsfig}
\usepackage[dvipsnames]{xcolor}
\usepackage{soul}
\usepackage{xfrac}
\usepackage[colorinlistoftodos,prependcaption,textsize=tiny]{todonotes}

\newcommand {\mm}[1] {\ifmmode{#1}\else{\mbox{\(#1\)}}\fi}

\newcommand {\scalprod}[2] {{\langle #1 , #2 \rangle}}
\newcommand{\denselist}{\itemsep 0pt\parsep=1pt\partopsep 0pt}

\newcommand{\ignore}[1]{}

\newsavebox{\smallProofsym}                 
\savebox{\smallProofsym}                               %
{
\begin{picture}(6,6)
\put(0,0){\framebox(6,6){}}
\put(0,2){\framebox(4,4){}}
\end{picture} 
}  

\makeatletter
\long\def\@makecaption#1#2{%
  \vskip\abovecaptionskip
  \internallinenumbers
  \sbox\@tempboxa{\small #1: #2}%
  \ifdim \wd\@tempboxa >\hsize
    \small #1: #2\par
  \else
    \global \@minipagefalse
    \hb@xt@\hsize{\hfil\box\@tempboxa\hfil}%
  \fi
  \vskip\belowcaptionskip}
\makeatother

\theoremstyle{plain}

\newtheorem*{supermaintheorem*}{Main Theorem}

\newtheorem*{supermaincorollary*}{Main Corollary}
\newtheorem*{supermaindefinition*}{Main Definition}

\newcommand{\Rspace}        {\mm{{\mathbb R}}}
\newcommand{\Sspace}        {\mm{{\mathbb S}}}
\newcommand{\Zspace}        {\mm{{\mathbb Z}}}

\newcommand{\Delaunay}[2]   {\mm{{\rm Del}_{#1}{({#2})}}}

\newcommand{\Hull}[2]       {\mm{{H}_{#1}{({#2})}}}

\newcommand{\bigOh}[1]      {\mm{\mathcal{O}\left({#1}\right)}}
\newcommand{\AAA}           {\mm{\bf A}}
\newcommand{\aaa}           {\mm{\bf a}}
\newcommand{\bbb}           {\mm{\bf b}}
\newcommand{\ccc}           {\mm{\bf c}}
\newcommand{\xxx}           {\mm{\bf{x}}}
\newcommand{\yyy}           {\mm{\bf{y}}}
\newcommand{\zzz}           {\mm{\bf{z}}}
\newcommand{\wdist}[2]      {\mm{\pi_{#1}{({#2})}}}

\newcommand{\sign}[1]       {\mm{\rm sgn}{({#1})}}

\newcommand{\cover}[2]      {\mm{{c}_{#1}{({#2})}}}
\newcommand{\card}[1]       {\mm{{\#}{#1}}}
\newcommand{\conv}[1]       {\mm{{\rm conv\,}{#1}}}

\newcommand{\norm}[1]       {\mm{\|{#1}\|}}
\newcommand{\Edist}[2]      {\mm{\|{#1}-{#2}\|}}

\newcommand{\Above}[1]      {\mm{{\rm abv}{({#1})}}}
\newcommand{\Through}[1]    {\mm{{\rm thr}{({#1})}}}
\newcommand{\Below}[1]      {\mm{{\rm blw}{({#1})}}}
\newcommand{\ee}            {\mm{\varepsilon}}

\DeclareMathOperator{\vol}{vol}

\definecolor{blue-green}{rgb}{0.0, 0.87, 0.87}
\definecolor{dark-green}{rgb}{0.0, 0.66, 0.00}

\newcommand{\Skip}[1]       {}

\newcommand{\AG}[1]         {{\textcolor{orange}{#1}}}

\title{On Spheres with $k$ Points Inside}
\titlerunning{On Spheres with $k$ Points Inside}

\author{Herbert Edelsbrunner}{IST Austria (Institute of Science and Technology Austria), Kloster\-neu\-burg, Austria}{herbert.edelsbrunner@ist.ac.at}{https://orcid.org/0000-0002-9823-6833}{}

\author{Alexey Garber}{School of Mathematical and Statistical Sciences, University of Texas Rio Grande Valley, Brownsville, Texas, USA}{alexey.garber@utrgv.edu}{https://orcid.org/0000-0002-9474-2077}{}

\author{Morteza Saghafian}{IST Austria (Institute of Science and Technology Austria), Kloster\-neu\-burg, Austria}{morteza.saghafian@ist.ac.at}{https://orcid.org/0000-0002-4201-5775}{}

\authorrunning{Edelsbrunner, Garber, Saghafian}

\Copyright{Edelsbrunner, Garber, Saghafian}



\funding{\footnotesize
  Work by the first and third authors is partially supported by the Wittgenstein Prize, Austrian Science Fund (FWF), grant no.\ Z 342-N31, and by the DFG Collaborative Research Center TRR 109, Austrian Science Fund (FWF), grant no.\ I 02979-N35.
  Work by the second author is partially supported by the Simons Foundation.
}

\begin{document}
\maketitle

\nolinenumbers

\begin{abstract}
  We generalize the classic definition of Delaunay triangulation and prove that for a locally finite and coarsely dense generic point set, $A \subseteq \Rspace^d$,
  the $d$-simplices whose vertices belong to $A$ and whose circumscribed spheres enclose exactly $k$ points of $A$ cover $\Rspace^d$ exactly $\binom{d+k}{d}$ times.
  Similarly, the subset of such simplices incident to a point in $A$ cover any small enough neighborhood of that point exactly $\binom{d+k-1}{d-1}$ times.
  We extend this result to the cases in which the points are weighted and when $A$ contains only finitely many points in $\Rspace^d$ or in $\Sspace^d$.
  Using these results, we give new proofs of classic results on $k$-facets, old and new combinatorial results for hyperplane arrangements, and a new proof for the fact that the volumes of hypersimplices are Eulerian numbers.
\end{abstract}


\section{Introduction}
\label{sec:1}

In the seminal paper \cite{Del34}, Boris Delaunay (also spelled Delone) introduced the Delaunay triangulation of a finite point sets using simplices with empty circumscribed spheres. 
His construction can be reformulated as follows: 
for a (finite and generic) point set, $A \subseteq \Rspace^d$, the simplices with vertices in $A$ that contain no points of $A$ inside their circumscribed spheres cover the convex hull of $A$ in one layer.
In this paper,\footnote{An earlier version of this paper appeared in \cite{EGS25}, which this version extends by generalizing the main theorem to weighted points in Section~\ref{sec:2.5} and to points and weighted points on the sphere in Sections~\ref{sec:3.1} and \ref{sec:3.2}, reproving a result by Clarkson about the maximum number of local minima in the $k$-level of a hyperplane arrangement in Section~\ref{sec:4.2}, and giving a tight bound on the maximum number of minimum helf chambers in a half-space arrangement in Section~\ref{sec:4.3}.}we generalize Delaunay's construction and prove similar properties for simplices with circumscribed spheres that enclose exactly $k$ points of $A$, for some fixed non-negative integer $k$.
We call these simplices the $k$-hefty simplices of $A$. 

\smallskip
We begin by introducing the main concepts.
A set $A \subseteq \Rspace^d$ is \emph{locally finite} if every closed ball contains at most a finite number of the points of $A$, and it is \emph{coarsely dense} if every closed half-space contains at least one and therefore infinitely many points of $A$.
If $A$ has both properties, we call it a \emph{thin Delone set}; compare with the more restrictive class of \emph{Delone sets}, which are \emph{uniformly discrete} and \emph{relatively dense}, meaning the smallest inter-point distance is bounded away from $0$, and the radius of the largest empty ball is bounded away from $\infty$.
We call $A$ \emph{generic} if no $d+1$ of its points lie on a common hyperplane, and no $d+2$ of its points lie on a common $(d-1)$-sphere.
Any $(d-1)$-sphere bounds a closed $d$-ball and thus partitions $\Rspace^d$ into points \emph{inside} the sphere (in the interior of the ball), points \emph{on} the sphere, and points \emph{outside} the sphere (in the complement of the closed ball).
Assuming $A$ is generic, there is a unique $(d-1)$-sphere passing through any $d+1$ points of $A$, which we call the \emph{circumscribed sphere} of the $d$-simplex spanned by the points.
\begin{supermaindefinition*}
  \label{dfn:hefty_simplex}
  Let $k$ be a non-negative integer and $A \subseteq \Rspace^d$ a generic thin Delone set or a generic finite set. 
  A $d$-simplex with vertices in $A$ is \emph{$k$-hefty} if exactly $k$ points of $A$ lie inside the circumscribed sphere of the $d$-simplex.
\end{supermaindefinition*}
For example, the $0$-hefty simplices are the top-dimensional simplices in the Delaunay triangulation of $A$,
and $k$-hefty simplices with $k > 0$ are related to the cells in higher-order Delaunay triangulations \cite{Aur90,EdOs21}.
Our main results are Theorems~\ref{thm:global_covering} and \ref{thm:local_covering}, which we restate here in less technical terms:
\begin{supermaintheorem*}
  Let $k$ be a non-negative integer, $A \subseteq \Rspace^d$ a generic thin Delone set, and $a \in A$.
  Then the $k$-hefty simplices of $A$ cover $\Rspace^d$ exactly $\binom{d+k}{d}$ times, and the $k$-hefty simplices incident to $a$ cover any sufficiently small neighborhood of $a$ $\binom{d+k-1}{d-1}$ times.
\end{supermaintheorem*}
More specifically, almost every point of $\Rspace^d$ is covered by exactly $\binom{d+k}{d}$ $k$-hefty simplices, while boundary points of $k$-hefty simplices are of course contained in more than this number of such simplices.
Similarly, almost every point sufficiently close to $a \in A$ is covered by exactly $\binom{d+k-1}{d-1}$ $k$-hefty simplices with vertex $a$, while boundary point are again contained in more than this number of such simplices.
We also prove versions of this theorem for finitely many points, weighted points, and points on the $d$-dimensional sphere.
In addition, we apply the results to get new proofs of old and new results on $k$-facets, arrangements of hyperplanes, arrangements of hemispheres, and the volume of hypersimplices.

\smallskip
The paper is organized as follows.
Section~\ref{sec:2} introduces the main definitions and proves the main result for thin Delone sets (Theorems~\ref{thm:global_covering} and \ref{thm:local_covering}), which it then extends to finite sets (Theorem~\ref{thm:covering_for_finite_sets}) and weighted points (Theorem~\ref{thm:covering_for_weighted_points}).
Section~\ref{sec:3} considers hefty simplices on the $d$-dimensional sphere, and generalizes the local and global covering results for points (Theorem~\ref{thm:covering_on_sphere})
and for weighted points (Theorem~\ref{thm:covering_on_sphere_weighted_case}).
Section~\ref{sec:4} applies the results on covering to give new proofs of old and new results in discrete geometry: a new proof of Alon and Gy\H{o}ry's bound on the number of at-most-$k$-sets (Proposition~\ref{prop:at_most_k-sets}), a new proof of a sharpened version of Lov\'{a}sz Lemma (Proposition~\ref{prop:exact_Lovasz_lemma}), a new proof of Clarkson's result on the number of minima in a hyperplane arrangement (Proposition~\ref{prop:counting_minima_in_levels}), a proof of a new bound on the maximum number of minimum heft cells in a half-space arrangement (Theorem~\ref{thm:repeated_minimum_heft}), and a new proof of the fact that the volumes of hypersimplices are Eulerian numbers (Theorem~\ref{thm:hypersimplex_identity}).
Section~\ref{sec:5} offers concluding remarks.

\section{Hefty Simplices in Euclidean Space}
\label{sec:2}

This section presents the main result of this paper, which we state for infinite sets and then extend to finite sets and sets of weighted points in Euclidean space.
We begin with the main technical lemma before stating and proving the main theorem.

\subsection{First a Technical Lemma}
\label{sec:2.1}

For technical reasons we first show that the $k$-hefty simplices of a thin Delone set $A$ are ``locally uniform'' in size.
Specifically, we prove an upper bound for the radii of spheres that enclose a fixed point, $x \in \Rspace^d$, as well as at most $k$ points of $A$.
To this end, we write $B(x,R)$ for the closed ball with center $x$ and radius $R$, and note that the number of points of $A$ in this ball goes to infinity when $x$ is fixed and $R$ goes to infinity.
\begin{lemma}
  \label{lem:hefty_simplices_have_bounded_radius}
  Let $A \subseteq \Rspace^d$ be coarsely dense, $k$ a non-negative integer, and $x \in \Rspace^d$.
  Then there exists $R = R(x,A,k)$ such that if $x$ is inside a sphere that is not fully contained in $B(x,R)$, then there are at least $k+1$ points of $A \cap B(x,R)$ inside this sphere.
\end{lemma}
\begin{proof}
  Without loss of generality, assume $x=0$. 
  For every unit vector, $u \in \Sspace^{d-1}$, the open half-space of points $y$ that satisfy $(y,u)>0$ contains infinitely many points of $A$. 
  It follows that the function $f_u \colon (0,\infty) \to \Zspace$ that maps $r > 0$ to the number of points of $A$ inside the sphere with center $ru$ and radius $r$ is non-decreasing and unbounded. 

  \smallskip
  We introduce $g \colon \Sspace^{d-1} \to \Rspace$ defined by $g(u) = \inf \{r > 0 \mid f_u(r)\geq k+1\}$ and claim that $g$ is bounded.
  To derive a contradiction, suppose $g$ is unbounded, and let $u_1, u_2, \ldots$ be an infinite sequence of unit vectors with $g(u_n) \geq n$.
  Since $\Sspace^{d-1}$ is compact, there is a subsequence that converges to a vector $u_0 \in \Sspace^{d-1}$.
  Let $S$ be the sphere with radius $g(u_0)+1$ and center $(g(u_0)+1) u_0$. 
  By construction, there are at least $k+1$ points of $A$ inside $S$.
  Since these points are (strictly) inside the sphere, there is a sufficiently small $\ee > 0$ such that moving the center of the sphere by at most $\ee$ while adjusting its radius so the origin remains on the sphere, retains at least $k+1$ points of $A$ inside the sphere.
  But this contradicts the unboundedness of $g$ as there are points $u_i$ in the subsequence that are within distance $\ee$ from $u_0$ with $g(u_i)$ much larger than $g(u_0)+1$.

  \smallskip
  Since $g$ is bounded, $M = \sup \{g(u) \mid u \in \Sspace^{d-1}\}$ is finite and, by construction of $g$, there are at least $k+1$ points of $A$ inside any sphere with center $y$ and radius $\norm{y}$ as long as $\norm{y} \geq M$.
  Setting $R = 2M$, every sphere with center $y$ that encloses the origin and is not contained in $B(0,R)$ has radius $r > M$.
  This sphere encloses the ball with center $M \frac{y}{\norm{y}}$ and radius $M$, so there are at least $k+1$ points of $A$ inside this sphere that all belong to $B(0,R)$.
\end{proof}
As an immediate consequence of Lemma~\ref{lem:hefty_simplices_have_bounded_radius}, the circumscribed sphere of any $k$-hefty simplex of $A$ that encloses $x$ is completely contained in $B(x,R)$, in which $R$ depends on $x$, $A$, and $k$.

\subsection{Global Covering}
\label{sec:2.2}

Our first goal is to generalize the classic result of Delaunay, which asserts that the $0$-hefty simplices of a generic thin Delone set cover $\Rspace^d$ in one layer; that is: every point of $\Rspace^d$ is contained in at least one $0$-hefty simplex and almost every point of $\Rspace^d$ is contained in exactly one $0$-hefty simplex.
Specifically, we show that for every generic thin Delone set, $A \subseteq \Rspace^d$, the $k$-hefty simplices cover $\Rspace^d$ $\binom{d+k}{d}$ times.
We call $\binom{d+k}{d}$ the \emph{$k$-th covering number} and note that it depends on the dimension, $d$, and the parameter, $k$, but not on the set $A$.
We call $x \notin A$ \emph{generic with respect to $A$} if $A \cup \{x\}$ is generic.
For a given generic thin Delone set $A$, almost every point $x \in \Rspace^d$ is generic with respect to $A$.
To see this, observe that by local finiteness of $A$ there are only countably many hyperplanes and spheres spanned by $d$ and $d+1$ points, respectively, so the union of these hyperplanes and spheres has Lebesgue measure zero.
\begin{theorem}
  \label{thm:global_covering}
  Let $k$ be a non-negative integer and $A \subseteq \Rspace^d$ a generic thin Delone set.
  Then any point $x \in \Rspace^d$ that is generic with respect to $A$ belongs to exactly $\binom{d+k}{d}$ $k$-hefty simplices of $A$.
\end{theorem}
\begin{proof}
  The case $d=1$ is obvious since every $k$-hefty simplex of $A$ is an interval with endpoints in $A$ and exactly $k$ points between the two endpoints.
  Every point that is generic with respect to $A$, i.e.\ in $\Rspace \setminus A$, is contained in exactly $k+1$ such intervals.
  For $d \geq 2$, the proof splits into three steps.

  \medskip \noindent
  \textsc{Step 1.} Letting $k$ be a non-negative integer and $A \subseteq \Rspace^d$ a generic thin Delone set, we prove that there is a constant such that any point generic with respect to $A$ is contained in exactly this constant number of $k$-hefty simplices of $A$.
  Write $\cover{k}{x,A}$ for the number of $k$-hefty simplices of $A$ that contain $x$. 
  By Lemma~\ref{lem:hefty_simplices_have_bounded_radius}, $\cover{k}{x,A}$ is finite.
  Indeed, every $k$-hefty simplex of $A$ that contains $x$ must select its vertices from the finitely many points inside the ball $B(x,2M)$.
  To show that $\cover{k}{x,A}$ is the same for all generic points, we move $x$ continuously from one point to another.
  The only time $\cover{k}{x,A}$ can change is when $x$ passes through the boundary of a $k$-hefty simplex. 
  It suffices to show that for every $(d-1)$-simplex, $\Delta$, with vertices in $A$, the number of $k$-hefty simplices with facet $\Delta$ is the same on both sides of $\Delta$. 

  \smallskip
  Consider the line, $L$, that consists of all points equidistant to the vertices of $\Delta$ and mark each point $y \in L$ with the number of points of $A$ inside the sphere with center $y$ that passes through the vertices of $\Delta$.
  This partitions $L$ into labeled intervals, and since $A$ is generic, the labels of two consecutive intervals differ by exactly one.
  Fix a \emph{left to right} direction on $L$, move $y$ in this direction, and observe that the portion of space inside the sphere centered at $y$ that lies to the left of the hyperplane spanned by $\Delta$ shrinks, while the portion to the right of this hyperplane grows.
  The transitions from an interval labeled $k+1$ to another labeled $k$ are in bijection with the $k$-hefty simplices with facet $\Delta$ to the left of $\Delta$.
  Indeed, as $y$ makes the transition, there is a point of $A$ that passes from inside to outside the sphere centered at $y$, so this point is to the left of $\Delta$.
  Similarly, the transitions from an interval labeled $k$ to another labeled $k+1$ are in bijection with the $k$-hefty simplices with facet $\Delta$ to the right of $\Delta$.
  There are equally many transitions of either kind because the labels go to infinity on both sides.
  This proves that $\cover{k}{x,A}$ does not depend on $x$.
  \begin{figure}[hbt]
    \centering
    \resizebox{!}{1.6in}{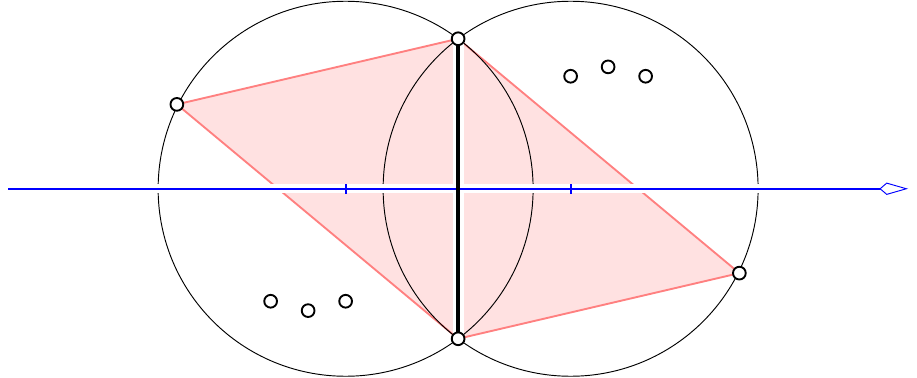}
    \caption{\footnotesize Two circles in the $1$-parameter family of circles that pass through the endpoint of the edge $\Delta$.
    Both are the circumcircles of $k$-hefty triangles, with $k = 3$ in the case displayed.
    As we move the center from left to right, every point that leaves the inside of the circle lies to the left of $\Delta$, and every point that enters the inside of the circle lies to the right of $\Delta$.}
    \label{fig:directedline}
  \end{figure}

  \medskip \noindent
  \textsc{Step 2.} We strengthen the claim proved in Step~1 by showing that the constant depends on $d$ and $k$ but not on $A$.
  Specifically, we prove that for every dimension, $d$, and non-negative integer, $k$, there exists a number $\cover{k}{d}$ such that for any generic thin Delone set, $A \subseteq \Rspace^d$, any point $x \in \Rspace^d$ generic with respect to $A$ belongs to exactly $\cover{k}{d}$ $k$-hefty simplices of $A$.

  \smallskip
  It suffices to show that for two thin Delone sets, $A$ and $A'$, and two points, $x, x'\in \Rspace^d$, that are generic with respect to both sets, $\cover{k}{x,A} = \cover{k}{x',A'}$.
  By Lemma~\ref{lem:hefty_simplices_have_bounded_radius}, there exists $R > 0$ such that if a sphere encloses $x$ and a point outside $B(x,R)$, then there are at least $k+1$ points of $A \cap B(x,R)$ inside this sphere.
  It follows that the circumscribed spheres of all $k$-hefty simplices that enclose $x$ are contained in $B(x,R)$. 
  Similarly, let $R'$ be the constant from Lemma~\ref{lem:hefty_simplices_have_bounded_radius} for $x'$ and $A'$.
  We construct a new thin Delone set, $A''$, by picking points $y$ and $y'$ at distance larger than $R+R'$ from each other, and letting $A'' \cap B(y,R)$ and $A'' \cap B(y',R')$ be translates of $A \cap B(x,R)$ and $A' \cap B(x',R')$, respectively.
  We perturb the points if necessary to achieve genericity, and add more points outside $B(y,R)$ and $B(y',R')$ until $A''$ is thin Delone.
  By choice of $R$ and $R'$ we have $\cover{k}{x,A}=\cover{k}{y,A''}$, because every $k$-hefty simplex of $A$ whose circumscribed sphere encloses $x$ translates into a $k$-hefty simplex of $A''$ whose circumscribed sphere encloses $y$, and vice versa. 
  Similarly, we have $\cover{k}{x',A'}=\cover{k}{y',A''}$. 
  Finally, $\cover{k}{x,A} = \cover{k}{x',A'}$ because $\cover{k}{y,A''} = \cover{k}{y',A''}$ as proved in Step~1.
  \begin{figure}[hbt]
    \centering
    \vspace{0.1in}
    \resizebox{!}{1.8in}{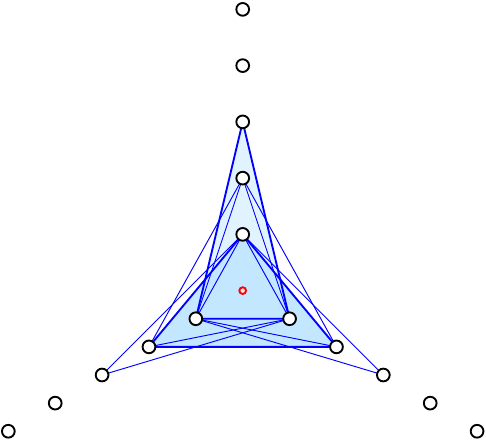}
    \caption{\footnotesize Before perturbation, the points of $A$ lie on three half-lines emanating from the origin.
    We show all six $2$-hefty triangles, and emphasize two of them by shading.}
    \label{fig:radial}
  \end{figure}

  \medskip \noindent
  \textsc{Step 3.} We provide an explicit example of a generic thin Delone set, $A \subseteq \Rspace^d$, and a point, $x\in \Rspace^d$, with $\cover{k}{d} = \binom{d+k}{d}$.
  Specifically, we prove that for every dimension $d$ and non-negative integer $k$, there exists a generic thin Delone set $A \subseteq \Rspace^d$, and a point $x$ generic with respect to $A$, such that $\cover{k}{x,A} = \binom{d+k}{d}$.

  \smallskip
  Consider a regular $d$-simplex with vertices $v_0, v_1, \ldots, v_d$ and barycenter $0 = \tfrac{1}{d+1} \sum_j v_j$, and
  $A$ the set of points of the form $iv_j$, for integers $i \geq 1$ and $j = 0, 1, \ldots, d$.
  We call $A$ a \emph{radial set} and the points with fixed $j$ a \emph{direction}; see Figure~\ref{fig:radial}.
  Set $x = 0$.
  Every $k$-hefty simplex of $A$ that contains $0$ has exactly one vertex in each direction.
  Letting $i_0 v_0, i_1 v_1, \ldots, i_d v_d$ be these vertices, the number of points inside the circumscribed sphere is $\sum_j (i_j - 1) = k$.
  Enumerating these simplices is the same as writing $k$ as an ordered sum of $d+1$ non-negative integers, and there are exactly $\binom{d+k}{d}$ ways to do that.
  To complete the proof, we perturb $A$ while making sure that $d+1$ points span a simplex that contains $0$ before the perturbation iff they span such a simplex after the perturbation.
  This completes the proof of the first claim in our Main Theorem.
\end{proof}

\subsection{Local Covering}
\label{sec:2.3}

By Theorem~\ref{thm:global_covering}, the $k$-hefty simplices of a thin Delone set cover $\Rspace^d$ without gap an integer number of times.
However, beyond one dimension, it is generally not possible to split such a cover into sub-covers (separate layers) that enjoy the same property.
As an example, consider Figure~\ref{fig:pentagon}, which shows one point surrounded by five others in $\Rspace^2$.
Assuming the six points are generic and the remaining points of the thin Delone set---which are not shown---are sufficiently far from the six, the point at the center is incident to five $1$-hefty triangles---which are shown.
Globally, the $1$-hefty triangles cover any small neighborhood of this point three times, but it is not possible to decompose this cover into three; see \cite{PTT07} for indecomposable coverings in other contexts.
It is however possible to split it locally into two covers.
The simplices that prevent a local split into three separate layers are all incident to a common vertex, and we prove that the hefty simplices incident to a common point cover any small neighborhood a fixed number of times.
This is the second claim in our Main Theorem.
\begin{figure}[hbt]
  \centering
  \vspace{0.1in}
  \resizebox{!}{1.8in}{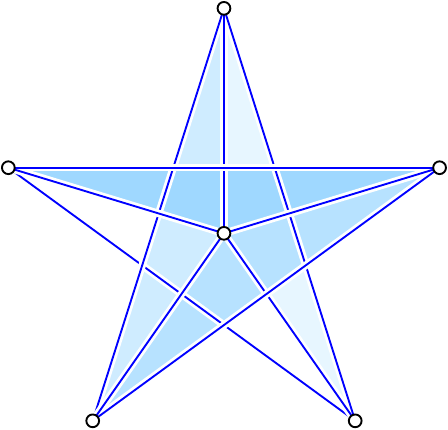}
  \caption{\footnotesize Six points of a Delone set, in which one is surrounded by five others, and the remaining points lie outside the circumcircles defined by these points.
  The five $1$-hefty triangles incident to the surrounded point cover any small neighborhood of that point twice, but they cannot be split into two disjoint covers of such a neighborhood.}
  \label{fig:pentagon}
\end{figure}
\begin{theorem}
  \label{thm:local_covering}
  Let $A \subseteq \Rspace^d$ be a generic thin Delone set, $k$ a non-negative integer, and $a \in A$ a point in this set.
  Then the $k$-hefty simplices of $A$ incident to $a$ cover any sufficiently small neighborhood of $a$ exactly $\binom{d+k-1}{d-1}$ times.    
\end{theorem}
\begin{proof}
  The case $k=0$ is that of the Delaunay triangulation of $A$.
  Its top-dimensional simplices cover $\Rspace^d$ exactly once, and thus also every neighborhood of $a \in A$ exactly once.
  We therefore assume $k > 0$. 
  Let $A' = A \setminus \{a\}$.
  Observe that $A'$ is also a generic thin Delone set and $a$ is generic with respect to $A'$.
  We can therefore apply Theorem~\ref{thm:global_covering} to the $(k-1)$-hefty simplices of $A'$ that contain $a \in \Rspace^d$.
  By Theorem~\ref{thm:global_covering}, there are $\binom{d+k-1}{d}$ such simplices and each of them contains a small neighborhood of $a$.

  \smallskip
  Let $y$ be a point in a sufficiently small neighborhood of $a$ that is inside the intersection of the $\binom{d+k-1}{d}$ $(k-1)$-hefty simplices of $A'$ containing $a$, and assume that $y$ is generic with respect to $A$.
  There are exactly $\binom{d+k}{d}$ $k$-hefty simplices of $A$ that contain $y$.
  These $\binom{d+k}{d}$ simplices split into those incident and not incident to $a$.
  Since $y$ is sufficiently close to $a$, the latter are the $(k-1)$-hefty simplices of $A'$ that contain $a$.
  The circumscribed sphere of each such simplex encloses exactly $k$ points of $A$ but only $k-1$ points of $A'$, which implies that the number of $k$-hefty simplices of $A$ that contain $y$ and are not incident to $a$ is $\binom{d+k-1}{d}$.
  It follows that the number of $k$-hefty simplices of $A$ that contain $y$ and are incident to $a$ is $\binom{d+k}{d}-\binom{d+k-1}{d}=\binom{d+k-1}{d-1}$, as claimed.
\end{proof}

\subsection{Finitely Many Points}
\label{sec:2.4}

For a finite set, $A \subseteq \Rspace^d$, the $k$-th covering number predicted by Theorem~\ref{thm:global_covering} applies sufficiently deep inside the set but acts only as an upper bound near the fringes.
To formalize this claim, we introduce a parametrized generalization of the convex hull: for each integer $k \geq 0$, the \emph{$k$-hull} of $A$, denoted $\Hull{k}{A}$, is the common intersection of all closed half-spaces that miss at most $k$ of the points in $A$.
Clearly, $\Hull{0}{A} = \conv{A}$, and $\Hull{k}{A} \supseteq \Hull{k+1}{A}$ for every $k$.
It is also easy to see that $\Hull{k}{A}$ is contained in the convex hull of $A$ after removing at most $k$ of its points, and indeed $\Hull{k}{A}$ \emph{is} the intersection of all such convex hulls.
By the Centerpoint Theorem of discrete geometry \cite[Section~4.1]{Ede87}, $\Hull{k}{A} \neq \emptyset$ if $k < \frac{n}{d+1}$, in which $n = \card{A}$ is the number of points in $A$.
\begin{theorem}
  \label{thm:covering_for_finite_sets}
  Let $A \subseteq \Rspace^d$ be finite and generic, $k \geq 0$ an integer, and $a \in A$.
  Then every $x \in \Rspace^d$ that is generic with respect to $A$ is covered by at most $\binom{d+k}{d}$ $k$-hefty simplices of $A$, with equality iff $x \in \Hull{k}{A}$, and if $x$ is sufficiently close to $a$, then it is covered by at most $\binom{d+k-1}{d-1}$ $k$-hefty simplices incident to $a$.
\end{theorem}
\begin{proof}
  To prove the two upper bounds, we add points outside all circumscribed spheres of $d+1$ points in $A$ to construct a thin Delone set $A' \subseteq \Rspace^d$.
  This is possible because the union of balls bounded by such spheres is bounded.
  Hence, $A \subseteq A'$, and any $k$-hefty simplex of $A$ is also a $k$-hefty simplex of $A'$.
  Assuming $x$ is generic with respect to $A'$, Theorems~\ref{thm:global_covering} and \ref{thm:local_covering} imply that exactly $\binom{d+k}{d}$ $k$-hefty simplices of $A'$ cover $x$, and exactly $\binom{d+k-1}{d-1}$ are incident to $a$.
  Since some of the $k$-hefty simplices of $A'$ may not be $k$-hefty simplices of $A$, these two numbers are only upper bounds.
  We continue by proving that $k$-th covering number is the correct number of $k$-hefty simplices of $A$ that contain $x$ iff $x \in \Hull{k}{A}$.
    
  \smallskip
  ``$\Longleftarrow$''.
  We prove $x \in \Hull{k}{A}$ implies that every $k$-hefty simplex of $A'$ covering $x$ is also a $k$-hefty simplex of $A$.
  Since $A'$ has exactly $\binom{d+k}{d}$ $k$-hefty simplices that cover $x$, this will imply that $A$ has the same number of such simplices.
  Consider such a $k$-hefty simplex of $A'$, and let $B \subseteq A'$ be the points inside its circumscribed sphere.
  Since $x \in \Hull{k}{A}$, every closed half-space that misses at most $k$ points of $A$ contains $x$.
  But $\card{B} = k$, so $x$ belongs to the convex hull of $A \setminus B$.
  The top-dimensional simplices of $\Delaunay{}{A \setminus B}$ cover the convex hull once, hence there is a unique $0$-hefty simplex of $A \setminus B$ that covers $x$.
  All points of $A' \setminus A$ lie outside its circumscribed sphere, which implies that the same simplex is the unique $0$-hefty simplex of $A'\setminus B$ that covers $x$.
  We have $B \subseteq A$ because the vertices of the simplex are in $A$, and all vertices we removed must come from $A$ as well.
  So this is a $k$-hefty simplex of $A$.

  \smallskip
  ``$\Longrightarrow$''.
  Assuming $x \not\in \Hull{k}{A}$, we modify the construction of $A'$ to show that there is a set $A'' \supseteq A$ such that all $k$-hefty simplices of $A$ are $k$-hefty simplices of $A''$ but at least one $k$-hefty simplex of $A''$ that covers $x$ is not a $k$-hefty simplex of $A$.
  In contrast to $A'$, the set $A''$ will be finite.
  Since $x \not\in \Hull{k}{A}$, there are $d$ points in $A$ such that the open half-space bounded by the hyperplane passing through these points that contains $x$ contains at most $k$ points of $A$.
  Let $\Delta$ be the $(d-1)$-simplex spanned by the $d$ points, write $\Delta^+$ for the open half-space that contains $x$, and let $B \subseteq A$ be the points in $\Delta^+$.
  By assumption, $\card{B} \leq k$.
  We get $A''$ by adding $1 + k - \card{B}$ points to $A$ as follows.
  First, we add $k-\card{B}$ points in $\Delta^+$ but outside all circumscribed spheres of $d+1$ points in $A$.
  After that, we add a point $y \in \Rspace^d$ so that the $d$-simplex that is the pyramid with apex $y$ and base $\Delta$ contains $x$, and all
  other $k$ points in $\Delta^+$ are inside the circumscribed sphere of this $d$-simplex. 
  We also require that $y$ is outside all circumscribed spheres of $d+1$ points in $A$.
  It is clear that such a point $y$ exists on a sufficiently large sphere that passes through all vertices of $\Delta$.
  As proved above, there are at most $\binom{d+k}{d}$ $k$-hefty simplices of $A''$ that cover $x$.
  The pyramid with apex $y$ and base $\Delta$ is such a simplex, but it is not a $k$-hefty simplex of $A$.
  Hence, the number of $k$-hefty simplices of $A$ that cover $x$ is strictly less than $\binom{d+k}{d}$, as claimed.
\end{proof}

Observe that the argument in part ``$\Longleftarrow$'' implies that once we fix a subset $B \subseteq A$ of size $k$, then the $k$-hefty simplices of $A$ that have exactly these points inside their circumscribed spheres have disjoint interiors.
This is because they are $0$-hefty simplices of $A \setminus B$.

\subsection{Weighted Points}
\label{sec:2.5}

This section extends the results on covering with $k$-hefty simplices from points to weighted points in $\Rspace^d$.
This more general setting has a long history within mathematics and appears among other topics in the study of packings and coverings \cite{Rog64}; see \cite{Aur87} for a survey.

\smallskip
To introduce the formalism, we let $\AAA \subseteq \Rspace^d \times \Rspace$ be a locally finite subset of $\Rspace^d$ in which each point is assigned a real weight.
We use notation to distinguish a \emph{weighted point}, $\aaa = (a, w_a)$, from its \emph{location}, $a \in \Rspace^d$, and we write $w_a \in \Rspace$ for its \emph{weight}.
The \emph{power} or \emph{weighted squared distance} of a point $x \in \Rspace^d$ from $\aaa$ is $\wdist{a}{x} = \Edist{x}{a}^2 - w_a$.
A convenient geometric picture of $\aaa$ is the sphere with center $a \in \Rspace^d$ and radius $r_a = \sqrt{w_a}$, which is either positive real, zero, or a positive real multiple of the imaginary unit.
Assuming $w_a > 0$, we have $\wdist{a}{x} = 0$ iff $x$ lies on this sphere.
Given two weighted points, $\aaa = (a, r_a^2)$ and $\bbb = (b, r_b^2)$\footnote{Because of the geometric picture with spheres, we find it more intuitive to talk about the radius rather than the weight, with the implicit understanding that the square of the radius can also be negative.} with $a\neq b$, the \emph{bisector} of $\aaa$ and $\bbb$ is the hyperplane of points $x \in \Rspace^d$ that satisfy
\begin{align}
    \wdist{a}{x} - \wdist{b}{x}
    &=2 \scalprod{b-a}{x} - \left( \norm{a}^2 - r_a^2 - \norm{b}^2 + r_b^2 \right)
    = 0 .
\end{align}
It is normal to $b-a$, but note that the point at which it intersects the line passing through $a$ and $b$ is not necessarily halfway between these points.
While the bisector of two weighted points is a $(d-1)$-plane (a hyperplane), the set of points at equal power from $p+1 \leq d+1$ weighted points is generically a $(d-p)$-plane.
The only requirement for the existence of this $(d-p)$-plane is that the $p+1$ locations do not lie on a common $(p-1)$-plane.
Specifically, for $p+1 = d+1$ weighted points with affinely independent locations, this set is a single point.

\begin{figure}[hbt]
  \centering
  \vspace{0.05in}
  \resizebox{!}{2.0in}{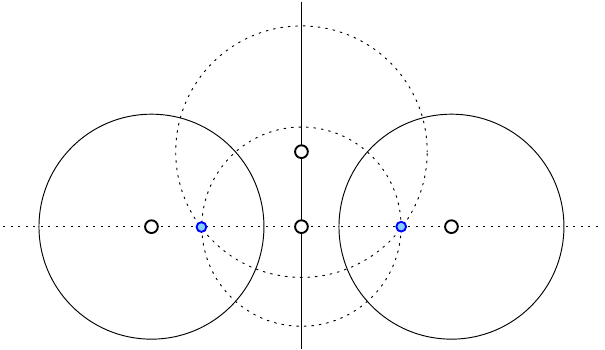}
  \caption{\footnotesize Given two weighted points, $\aaa_0$ and $\aaa_1$, the circles of all weighted points orthogonal to both pass through the \emph{blue} $0$-sphere (pair of points) on the line that passes through $a_0$ and $a_1$.}
  \label{fig:pencil}
\end{figure}
We call $\aaa = (a, r_a^2)$ and $\bbb = (b, r_b^2)$ \emph{orthogonal} to each other if $\Edist{a}{b}^2 = r_a^2 + r_b^2$.
By construction, for every point $x$ with equal power from $p+1 \leq d+1$ weighted points, $\aaa_0, \aaa_1, \ldots, \aaa_p$, there is a unique weight, $w_x = r_x^2$, such that $\xxx = (x, r_x^2)$ is orthogonal to each $\aaa_i$, for $0 \leq i \leq p$.
Let $z$ be the point that minimizes the common power from $\aaa_0, \aaa_1, \ldots, \aaa_p$, and note that it is the unique point at which the $(d-p)$-plane of points $x$ with equal power from $a_0,a_1,\ldots,a_p$ intersects the $p$-plane that passes through $a_0, a_1, \ldots, a_p$.
Let $S$ be the $(p-1)$-sphere in which the sphere of $\zzz = (z, r_z^2)$ intersects the $p$-plane.
Importantly, the sphere of every weighted point orthogonal to each $\aaa_i$, for $0 \leq i \leq p$, passes through $S$; see Figure~\ref{fig:pencil} for an illustration for $p+1 = 2$ weighted points in $\Rspace^2$.
We formally state a generalization of this claim that allows for non-positive weights, and prove it for later reference.
\begin{lemma}
  \label{lem:orthogonal_pencil}
  Let $\aaa_0, \aaa_1, \ldots, \aaa_p$ be $p+1 \leq d+1$ generic weighted points in $\Rspace^d$, and $\zzz = (z, r_z^2)$ the weighted point orthogonal to $\aaa_i$, for $0 \leq i \leq p$, with minimum weight.
  Then $\Edist{x}{z}^2 = r_x^2 - r_z^2$ for every weighted point $\xxx = (x, r_x^2)$ that is orthogonal to $\aaa_i$ for $0 \leq i \leq p$.
\end{lemma}
\begin{proof}
  By definition, $x$ and $z$ both lie on the $(d-p)$-plane of points with equal power from the $p+1$ weighted points.
  This $(d-p)$-plane is normal to $b-a$, which implies that $x, z, a_i$ are the vertices of a right-angled triangle, so $\Edist{x}{a_i}^2 = \Edist{x}{z}^2 + \Edist{z}{a_i}^2$ for each $0 \leq i \leq p$.
  To avoid double-indices, write $r_i^2$ for the weight of $\aaa_i$.
  Since $\xxx$ and $\zzz$ are both orthogonal to $\aaa_i$, we also have $\Edist{x}{a_i}^2 = r_x^2 + r_i^2$ and $\Edist{z}{a_i}^2 = r_z^2 + r_i^2$.
  Plugging the last two relations into the first, we get $\Edist{x}{z}^2 = \Edist{x}{a_0}^2 - \Edist{z}{a_0}^2 = r_x^2 - r_z^2$, as claimed.
\end{proof}

We now return to main topic of this section.
Writing $A \subseteq \Rspace^d$ for the projection of $\AAA \subseteq \Rspace^d \times \Rspace$, we call $\AAA$ a \emph{thin Delone set of weighted points} if $A$ is thin Delone and the absolute weights of the points are bounded away from $\infty$.
Furthermore, $\AAA$ is \emph{generic} if $A$ is generic and for every $d+1$ weighted points there is a unique and distinct weighted point that is orthogonal to each of the $d+1$ weighted points.
The $d$-simplex defined by $d+1$ weighted points is the convex hull of their locations, and the \emph{heft} of this simplex is the number of weighted points $\bbb = (b, r_b^2) \in \AAA$ that satisfy $\Edist{x}{b}^2 - r_x^2 - r_b^2$ < 0, in which $\xxx = (x, r_x^2)$ is the unique weighted point orthogonal to the $d+1$ weighted points.
\begin{theorem}
  \label{thm:covering_for_weighted_points}
  Let $\AAA \subseteq \Rspace^d \times \Rspace$ be a generic thin Delone set of weighted points, and $a$ the projection of a weighted point $\aaa \in \AAA$ to $\Rspace^d$.
  Then the simplices with heft $k$ cover $\Rspace^d$ exactly $\binom{d+k}{d}$ times, and the simplices with heft at most $k$ incident to $a$ cover every sufficiently small neighborhood of $a$ at most $\binom{d+k}{d}$ times.
\end{theorem}
\begin{proof}
  The proof of the global covering constant is similar to that of Theorem~\ref{thm:global_covering}, and we present abridged versions of its three steps.
  Let $x \in \Rspace^d$ be generic with respect to $A$, and write $\cover{k}{x,\AAA}$ for the number of $k$-hefty simplices of $\AAA$ that contain $x$.

  \smallskip
  In Step~I we show that $\cover{k}{x,\AAA}$ does not depend on $x$.
  Indeed, if we move $x$ continuously it can possibly change only when $x$ crosses a facet of some $k$-hefty simplex, but even then $\cover{k}{x,\AAA}$ is preserved, as we now show.
  Let $\Delta$ be such a facet and $L$ the line of points with equal power from the weighted points that project to the vertices of $\Delta$.
  For each generic point $y \in L$, mark $y$ with the number of weighted points $\bbb \in \AAA$ such that $\wdist{b}{y} < \wdist{a}{y}$ for each vertex $a$ of $\Delta$.
  Equivalently, $\Edist{b}{y} < r_b^2 + r_y^2$, in which $\yyy = (y, r_y^2)$ is the unique sphere with center $y$ that is orthogonal to the spheres whose centers are the vertices of $\Delta$.
  This splits $L$ into open segments marked with integer values, and the number of $k$-hefty simplices that share $\Delta$ is the number of segments labeled $k$.
  Now use Lemma~\ref{lem:orthogonal_pencil} while walking along $L$ from left to right.
  A $k$-hefty simplex with facet $\Delta$ lies to the left of $\Delta$ if it corresponds to a transition from a segment marked $k+1$ on the left to one marked $k$ on the right, and it lies to the right of $\Delta$ if the transition is the other way round.
  Since there are equally many transitions of either kind, there are equally many $k$-hefty simplices that share $\Delta$ on both sides.
  Hence, $\cover{k}{x,\AAA}$ is independent of $x$.

  \smallskip
  In Step~II we show that $\cover{k}{x,\AAA}$ does not depend on $\AAA$ either, only on the dimension of the ambient space.
  This step requires a generalization of Lemma~\ref{lem:hefty_simplices_have_bounded_radius} to weighted points, which is easy because all absolute weights are bounded by a constant.
  We can therefore take bounded neighborhoods of two points, $x$ and $x'$, place them without overlap next to each other, and prove $\cover{k}{x,\AAA} = \cover{k}{x',\AAA'}$ as in the first step.
  In Step III we re-use the radial set of points displayed in Figure~\ref{fig:radial}.
  The unweighted points are weighted points with weight $0$, so we finally conclude that $\cover{k}{x,\AAA} = \binom{d+k}{d}$, as claimed.

  \smallskip
  Finally, we prove the inequality for local covering.
  After removing $\aaa$ from $\AAA$, the global covering constant proved above implies that for every $0 \leq j \leq k-1$, exactly $\binom{d+j}{d}$ $j$-hefty simplices of $\AAA \setminus \{ \aaa \}$ contain the point $a$.
  Adding $\aaa$ back, the heft of each of these simplices increases by at most $1$, so the number of simplices of heft at most $k$ that contain $a$ is at least
  \begin{align}
    M_{k-1} &= \tbinom{d+0}{d} + \tbinom{d+1}{d} + \ldots + \tbinom{d+k-1}{d} .
  \end{align}
  Indeed, it is also possible that a $k$-hefty simplex of $\AAA \setminus \{ \aaa \}$ is only $k$-hefty for $\AAA$, but this does not invalidate the derived inequality.
  The simplices of heft at most $k$ cover $\Rspace^d$ exactly $M_k$ times, so the subset of simplices that are incident to $a$ cover any sufficiently small neighborhood of $a$ at most $M_k - M_{k-1} = \binom{d+k}{d}$ times.
\end{proof}

The upper bound for local covering in Theorem~\ref{thm:covering_for_weighted_points} is tight, as we will see shortly.
Indeed, if there were a smaller upper bound, then its application to proving an upper bound for $\Psi_d (k,n)$ in the next subsection would contradict the lower bound illustrated in Figure~\ref{fig:hexmesh}.
In other words, that lower bound construction can be turned into a lower bound construction for the local covering question, and we leave the details to the interested reader.

\section{Hefty Simplices on the Sphere}
\label{sec:3}

In this section, we extend the results of Section~\ref{sec:2} from the Euclidean to the spherical setting.
To a large extent, the proofs are similar to those in Section~\ref{sec:2}, so we limit ourselves to highlighting the differences.

\subsection{(Unweighted) Points}
\label{sec:3.1}

The results in Theorems~\ref{thm:global_covering} and \ref{thm:local_covering} can be extended to finite sets in $\Sspace^d$.
Note however that in the spherical setting, there is an ambiguity in the definition of a $k$-hefty simplex since any $(d-1)$-sphere on the $d$-sphere bounds two complementary $d$-balls.
We thus require that the finite set satisfies the following condition.
\begin{definition}
  \label{dfn:balanced_sets}
  A set $A \subseteq \Sspace^d$ is \emph{$k$-balanced} if every open hemisphere contains at least $k+1$ points of $A$.
\end{definition}
With this assumption, only the smaller of the two open $d$-balls has a chance to contain $k$ and no more of the points, so we consider it the \emph{inside} of the $(d-1)$-sphere.
The notions of genericity and heft are straightforward adaptions of those introduced in Euclidean space.
In particular, a finite set $A \subseteq \Sspace^d$ is \emph{generic} if no $d+1$ points belong to a common $(d-1)$-dimensional great-sphere, and no $d+2$ points belong to a common $(d-1)$-dimensional sphere in $\Sspace^d$.
Equivalently, no $d+1$ of the points (unit vectors) are linearly dependent, and no $d+2$ are affinely dependent in $\Rspace^{d+1}$.
Furthermore, a point $a \in \Sspace^d$ is \emph{generic with respect to} $A$ if $A \cup \{a\}$ is generic.
The \emph{heft} of a spherical $d$-simplex spanned by $d+1$ points in $A$ is the number of points inside the unique $(d-1)$-sphere that passes through the $d+1$ points, and we call it a \emph{$k$-hefty simplex} of $A$ if this number is $k$.
\begin{theorem}
  \label{thm:covering_on_sphere}
  Let $k$ be a non-negative integer, $A \subseteq \Sspace^d$ a finite generic $k$-balanced set, and $a \in A$.
  Then the $k$-hefty simplices of $A$ cover $\Sspace^d$ exactly $\binom{d+k}{k}$ times, and the subset of these simplices incident to $a$ cover every sufficiently small neighborhood of $a$ exactly $\binom{d+k-1}{d-1}$ times.
\end{theorem}
\begin{proof}
  We first prove the claim about global covering.
  Let $x, x' \in \Sspace^d$ be antipodal points on the sphere.
  We denote the hyperplane that touches $\Sspace^d$ in $x'$ by $\Rspace^d$.
  After stereographic projection from $x$ onto $\Rspace^d$, the set $A$ transforms into a finite set $A' \subseteq \Rspace^d$, and since $A$ is $k$-balanced, we have $x' \in \Hull{k}{A'}$.
  Furthermore, the stereographic projection maps $k$-hefty simplices of $A$ that contain $x' \in \Sspace^d$ bijectively (as vertex sets) to $k$-hefty simplices of $A'$ that contain $x' \in \Rspace^d$.
  By Theorem~\ref{thm:covering_for_finite_sets}, $A'$ has exactly $\binom{d+k}{d}$ $k$-hefty simplices of this kind, which implies that $A$ has the same number of $k$-hefty simplices that contain $x'$.

  \smallskip
  We second prove the claim about local covering.
  Note that removing $a$ gives a $(k-1)$-balanced set $A\setminus\{a\} \subseteq \Sspace^d$. 
  Using this, the proof proceeds verbatim as in Theorem~\ref{thm:local_covering}.
\end{proof}

\subsection{Weighted Points}
\label{sec:3.2}

We will need bounds on global and local covering with hefty simplices also for weighted points on the sphere.
We begin by extending the framework of weighted points from $\Rspace^d$ to $\Sspace^d$.
It is convenient to work with weights $w \in [-1, \infty)$ and corresponding radii $r = \sqrt{w+1}$, which are non-negative reals.
In particular, we specify a point of $\Rspace^{d+1}$ in polar coordinates, $\aaa = (a, r_a) \in \Sspace^d \times [0,\infty)$, call $\aaa$ a \emph{weighted point}, $a \in \Sspace^d$ its \emph{location}, and $w_a = r_a^2 - 1$ its \emph{weight}.
Note that $\aaa$ is on $\Sspace^d$ for $w_a = 0$, inside $\Sspace^d$ for $w_a < 0$, and outside $\Sspace^d$ for $w_a > 0$.
Similar to the Euclidean setting, we introduce the \emph{power} or \emph{weighted squared distance} of a point $x \in \Sspace^d$ from $\aaa$ as $\wdist{a}{x} = \Edist{x}{\aaa}^2 - w_a = \Edist{x}{r_a a}^2 - r_a^2 + 1$.
Given two weighted points, $\aaa = (a, w_a)$ and $\bbb = (b, w_b)$ with $a\neq b$, their bisector is the unique $(d-1)$-dimensional great-sphere contained in the $d$-plane with normal $\bbb - \aaa$ that passes through the origin of $\Rspace^{d+1}$.
Indeed, a point $x \in \Sspace^d$ belongs to the bisector iff
\begin{align}
  \wdist{a}{x} - \wdist{b}{x}
    &= \Edist{x}{\aaa}^2 - r_a^2 + 1 - \Edist{x}{\bbb}^2 + r_b - 1
    = 2 \scalprod{x}{\bbb - \aaa} = 0,
  \label{eqn:bisector}
\end{align}
which describes the mentioned great-sphere.
Like in $\Rspace^d$, for every weighted point there are infinitely many second weighted points that define the same bisector.
Indeed, if we are given $\aaa$ and any $(d-1)$-dimensional great-sphere of $\Sspace^d$, then this great-sphere is the bisector of $\aaa$ and $\bbb$ whenever $\bbb - \aaa$ is normal to the corresponding $d$-plane.
Furthermore, the direction of $\bbb - \aaa$ along the normal line decides on which side of this great-sphere the points have smaller power from $\bbb$ than from $\aaa$.

\smallskip
The heft of $d+1$ weighted points is defined in terms of the weighted points separated from the origin by the $d$-plane that passes through them.
Before we formalize this notion, we call a finite set $\AAA \subseteq \Sspace^d \times [0, \infty)$ \emph{$k$-balanced} if there are at least $k+1$ weighted points on each side of every $d$-plane that passes through the origin.
\begin{definition}
  \label{dfn:hefty_simplex_weighted_case}
  Let $\AAA \subseteq \Sspace^d \times [0, \infty)$ be a generic $k$-balanced set of weighted points.
  We call the spherical $d$-simplex spanned by $d+1$ locations a \emph{$k$-hefty simplex} of $\AAA$ if the $d$-plane that passes through the corresponding $d+1$ weighted points separates exactly $k$ of the remaining points in $\AAA$ from $0 \in \Rspace^{d+1}$.
\end{definition}
As illustrated in Figure~\ref{fig:htriangle}, the $d+1$ locations span a spherical simplex, which is the central projection of the $d$-simplex spanned by the corresponding $d+1$ weighted points, which are points in $\Rspace^{d+1}$.
\begin{figure}[hbt]
  \centering \vspace{0.13in}
  \resizebox{!}{2.3in}{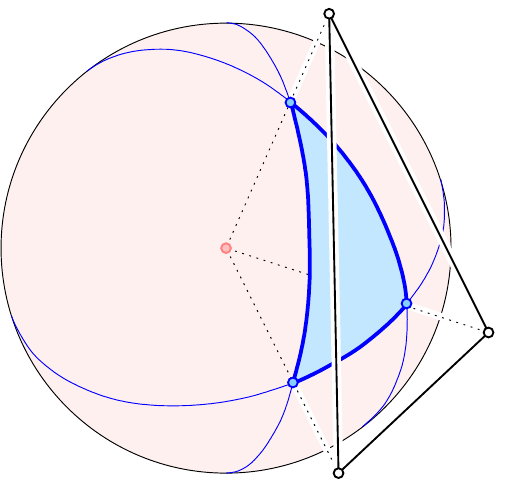}
  \vspace{0.03in}
  \caption{\footnotesize The spherical triangle spanned by $a, b, c \in \Sspace^2$ is the central projection of the Euclidean triangle spanned by the corresponding points $\aaa, \bbb, \ccc \in \Rspace^3$.
  The spherical triangle is $k$-hefty if $k$ of the weighted points lie beyond the plane of the Euclidean triangle.}
  \label{fig:htriangle}
\end{figure}
\begin{theorem}
  \label{thm:covering_on_sphere_weighted_case}
  Let $\AAA \subseteq \Sspace^d \times [0, \infty)$ be a generic $k$-balanced set of finitely many weighted points, and $\aaa \in \AAA$.
  Then the spherical $d$-simplices with heft $k$ cover $\Sspace^d$ exactly $\binom{d+k}{d}$ times, and the spherical $d$-simplices with heft at most $k$ incident to $a$ cover every sufficiently small neighborhood of this point at most $\binom{d+k}{d}$ times.
\end{theorem}
More precisely, if a point $x \in \Sspace^d$ does not lie on a common $(d-1)$-dimensional great-sphere with $d$ locations, then it lies inside exactly $\binom{d+k}{d}$ $k$-hefty simplices of $\AAA$.

\begin{proof}
  The proof is similar to that of Theorem~\ref{thm:covering_for_weighted_points}, so we limit ourselves to highlighting the differences.
  Step~I---which proves that $\cover{k}{x,A}$ does not depend on $x$---is essentially unchanged, except that we move along a directed great-circle instead of a directed line.
  An important technical ingredient is the generalization of Lemma~\ref{lem:orthogonal_pencil}, which we omit.

  \smallskip
  Step~II is different than in Theorem~\ref{thm:covering_for_weighted_points}, so we spell it out in full.
  Let $\AAA, \AAA' \subseteq \Sspace^d \times [0,\infty)$ be two $k$-balanced sets of weighted points, and assume that they are generic with respect to each other. 
  Letting $\aaa' \in \AAA'$, we claim that $\cover{k}{x, \AAA} = \cover{k}{x, \AAA \cup \{\aaa'\}}$ because the $k$-hefty simplices of $\AAA$ and $\AAA \cup \{\aaa'\}$ that contain $- a'$ are the same.
  Adding the remaining points of $\AAA'$ one by one to $\AAA$, we get $\cover{k}{x, \AAA} = \cover{k}{x, \AAA \cup \AAA'}$ and, by symmetry, $\cover{k}{x, \AAA'} = \cover{k}{x, \AAA' \cup \AAA}$.
  Hence, $\cover{k}{x, \AAA}$ is also independent of $\AAA$, and we write $\cover{k}{d}$.

  \smallskip
  Step~III is again similar, and we just need to adapt the radial set of Figure\ref{fig:radial} to the spherical setting by making it $k$-balanced, with a sufficiently large number of points in each of the $d+1$ directions.
  This implies $\cover{k}{d} = \binom{d+k}{d}$, as claimed.
  Finally, the proof of the inequality for local covering is as in Theorem~\ref{thm:covering_for_weighted_points}, and we omit details.
\end{proof}
Similar to the Euclidean setting, the upper bound for local covering in Theorem~\ref{thm:covering_on_sphere_weighted_case} is tight, and we refer to Figure~\ref{fig:hexmesh} and the proof of Theorem~\ref{thm:repeated_minimum_heft} in Section~\ref{sec:4.3} for further information.
We furthermore note that Step~II in the proof of Theorem~\ref{thm:covering_on_sphere_weighted_case} works verbatim in the proof of Theorem~\ref{thm:covering_on_sphere}.
However, in order to utilize the stereographic projection for the weighted points, we need to describe how the weights of the points change.

\section{Applications}
\label{sec:4}

In this section, we discuss four applications of the theorems in Sections~\ref{sec:2} and \ref{sec:3}:
the connection to $k$-sets and $k$-facets in Section~\ref{sec:4.1}, counting the minima in hyperplane arrangements in Section~\ref{sec:4.2}, counting cells with minimum heft in hemisphere arrangements in Section~\ref{sec:4.3}, and the relation to the volume of hypersimplices and Eulerian numbers in Section~\ref{sec:4.4}.

\subsection{$k$-sets and $k$-facets}
\label{sec:4.1}

In this section, we briefly discuss connections between the $k$-simplices studied in the preceding two sections and the $k$-sets and $k$-facets, a classical topic in discrete geometry; see e.g.\ \cite[Chapter~11]{Mat02}.
Letting $A$ be a generic set of $n$ points in $\Rspace^d$, a \emph{$k$-set} is a subset of $k \leq n$ points, $B \subseteq A$, such that $B$ and $A \setminus B$ can be separated by a hyperplane.
Related to this is the notion of a \emph{$k$-facet}, which is a subset $D \subseteq A$ of $d$ points such that the hyperplane that passes through the points in $D$ partitions $A \setminus D$ into $k$ and $n-d-k$ points on its two sides.
It will be convenient to overload the term and to refer to $\Delta = \conv{D}$ as a $k$-facet as well.
We refer to \cite[Section~2.2]{Wag08} for a discussion of the relation between $k$-sets and $k$-facets.

\smallskip
We present alternative proofs of two well-established results, which we state in terms of $k$-facets.
Both proofs make use of the \emph{inversion} of $A$ through the unit sphere centered at a point $z \in \Rspace^d \setminus A$, which maps a point $a \in \Rspace^d \setminus \{z\}$ to $\iota_z (a) = z + (a-z)/\Edist{a}{z}^2$.
It is not difficult to see that the image of a hyperplane that avoids $z$ is a $(d-1)$-dimensional sphere that passes through $z$, and the image of the open half-space that does not contain $z$ is the open ball bounded by this sphere.
If the hyperplane passes through the points of a $k$-facet, $D \subseteq A$, and separates $z$ from the $k$-set on the other side, then the $d$-simplex spanned by $z$ and the points in $\iota_z (D)$ is a $k$-hefty simplex of $A' = \iota_z (A) \cup \{z\}$ incident to $z$.

\smallskip
Write $F_k (A)$ for the number of $k$-facets of $A$.
We first reprove the following $2$-dimensional result by Alon and Gy\H{o}ri \cite{AlGy86} for $k$ less than about a third of the number of points.
\begin{proposition}
  \label{prop:at_most_k-sets}
  Let $A$ be a generic set of $n$ points in $\Rspace^2$, and $k \leq \frac{n-3}{3}$ a non-negative integer.
  Then $\sum_{j=0}^k F_j(A) \leq (k+1) n$.
\end{proposition}
\begin{proof}
  Recall that the $k$-hull of $A$ is the intersection of all closed half-spaces that miss at most $k$ points of $A$, denoted $\Hull{k}{A}$.
  By the Centerpoint Theorem of discrete geometry, $\Hull{k}{A}$ has a non-empty interior if $k \leq \frac{n-3}{3}$; see e.g.\ \cite[Section~4.1]{Ede87}.
  Let $z \not\in A$ be a point in the interior $\Hull{k}{A}$, and set $A' = \iota_z (A) \cup \{z\}$.
  Let $j \leq k$.
  As explained above, the inversion through the unit circle centered at $z$ maps every $j$-facet of $A$ to a $j$-hefty triangle of $A'$ incident to $z$.
  By Theorem~\ref{thm:covering_for_finite_sets}, the $j$-hefty triangles incident to $z$ cover every sufficiently small neighborhood of $z$ at most $j+1$ times, so if we consider all $j$ between $0$ and $k$, we get the intersection of these neighborhoods covered at most $1+2+\ldots+(k+1) = \binom{k+2}{2}$ times.

  \smallskip
  To continue, we draw a half-line emanating from $z$ through every point in $\iota_z (A \setminus \{z\})$, thus splitting the neighborhood of $x$ into $n$ angles.
  Every hefty triangle incident to $x$ covers a contiguous sequence of these angles, and for each $i \geq 1$, there are at most $n$ triangles that cover exactly $i$ of these angles.
  Each angle is covered some integer number of times, and we take the sum of these numbers over all angles.
  If $\sum_{j=0}^k F_j(A) > (k+1)n$, then this sum is strictly greater than $n [1+2+\ldots+(k+1)] = n \binom{k+2}{2}$, which implies that one of the angles is covered more than $\binom{k+2}{2}$ times.
  But this contradicts Theorem~\ref{thm:covering_for_finite_sets}.
\end{proof}

We continue with Lov\'{a}sz Lemma---see \cite{BFL90} but also \cite[Lemma~11.3.2]{Mat02}---which is a crucial ingredient in many arguments about $k$-sets; see e.g.\ \cite{Rub24}.
This lemma gives an asymptotic upper bound on the number of $k$-facets a line can intersect.
We give a short proof of a more general version by Welzl~\cite{Wel01}.
To state the lemma, we say a directed line, $L$, \emph{enters} a $k$-facet $\Delta$ of $A$ if $L$ intersects $\Delta$ at an interior point and moves from the side with $k$ points to the side with $n-d-k$ points as it passes through the $(d-1)$-simplex.
\begin{proposition}
  \label{prop:exact_Lovasz_lemma}
  Let $A \subseteq \Rspace^d$ be a generic finite set, $k$ a non-negative integer, and $L$ a directed line. 
  Then $L$ enters at most $\binom{d+k-1}{d-1}$ $k$-facets of $A$.
\end{proposition}
\begin{proof}
  We may assume that $L$ passes through the convex hull of $A$.
  Let $z \in L$ be a point outside $\conv{A}$ that we reach after passing through the entire convex hull.
  The inversion through the unit sphere centered at $z$ maps every $k$-facet entered by $L$ to a $k$-hefty simplex incident to $z$ such that $L$ passes through this simplex before it reaches $z$.
  By Theorem~\ref{thm:covering_for_finite_sets}, there are at most $\binom{d+k-1}{d-1}$ such $k$-hefty simplices.
\end{proof}
Since a line can be directed in two different ways, the number of $k$-facets of $A$ that can be entered by a line on one or the other direction is at most $2 \binom{d+k-1}{d-1}$.

\subsection{Counting Local Minima in Levels}
\label{sec:4.2}

In this section, we give a new proof of a combinatorial result on arrangements due to Clarkson~\cite{Cla93}.
To introduce the question, we think a (non-vertical) $d$-plane in $\Rspace^{d+1}$ as the graph of an affine function, $h \colon \Rspace^d \to \Rspace$, defined by mapping $x \in \Rspace^d$ to $h(x) = \scalprod{x}{g} + c$, in which $g \in \Rspace^d \setminus \{0\}$ is the \emph{gradient} and $c \in \Rspace$ is the \emph{offset}.
We say the $d$-plane passes \emph{above}, \emph{through}, and \emph{below} a point $\xxx = (x, \alpha) \in \Rspace^d \times \Rspace$ if $\alpha - h(x)$ is negative, zero, and positive, respectively.
Consider a finite collection of such $d$-planes, and write $\Above{\xxx}$, $\Through{\xxx}$, and $\Below{\xxx}$ for the number of $d$-planes that pass above, through, and below $\xxx$.
For each integer $k$, the \emph{$k$-level} of the collection is the piecewise linear function $L_k \colon \Rspace^d \to \Rspace$ that maps $x \in \Rspace^d$ to $\xi = L_k(x)$ such that $\Above{\xxx} \leq k < \Above{\xxx} + \Through{\xxx}$ for $\xxx = (x, \xi)$.
In words, $\xi$ is the $(k+1)$-largest value assigned to $x$ by the affine functions in the collection.

\smallskip
We call a point $x \in \Rspace^d$ a \emph{local minimum} of $L_k$ if it has an open neighborhood with $L_k(x) < L_k(y)$ for every $y \neq x$ in this neighborhood.
The question answered by Clarkson~\cite{Cla93} is how many local minima there can be.
The upper bound he proves is tight and independent of the number of $d$-planes in the collection.
Incidentally, it can be used to deduce the celebrated Upper Bound Theorem for convex polytopes; see \cite{ANPS93} for another corollary of Clarkson's result and \cite{Zie95} for background on this general topic.
\begin{proposition}
  \label{prop:counting_minima_in_levels}
  Let $k$ be a non-negative integer and $L_k \colon \Rspace^d \to \Rspace$ the $k$-level of a finite collection of affine functions.
  Then $L_k$ has at most $\binom{d+k}{d}$ local minima.
\end{proposition}
\begin{proof}
  We begin by fixing a point $\zzz = (0, \zeta) \in \Rspace^d \times \Rspace$ such that all $d$-planes in the given collection pass below $\zzz$.
  Furthermore, we index the $d$-planes arbitrarily and write $h_i(x) = \scalprod{x}{g_i} + c_i$
  for the $i$-th affine function.
  Since we are after the maximum number of local minima, we may assume a generic collection, which includes that at most $d+1$ affine functions agree on any one point in $\Rspace^d$.

  \smallskip
  Suppose $x \in \Rspace^d$ is a local minimum.
  Then there are $d+1$ affine functions that map $x$ to $\xi = L_k(x)$.
  Since $x$ is a local minimum, their gradients span all directions, by which we mean that their convex hull is a $d$-simplex that contains $0 \in \Rspace^d$ in its interior.
  Next we map the $i$-th affine function to the point $\yyy_i = (y_i, \eta_i)$ with $y_i = 2 \lambda_i g_i$ and $\eta_i = \zeta - 2 \lambda_i$, in which $\lambda_i = ({\zeta - c_i}) / ({\norm{g_i}^2 + 1})$.
  We note that $\lambda_i > 0$ because $\zeta > c_i$.
  Importantly, the graph of $h_i$ is the bisector of $\zzz$ and $\yyy_i$, which is not difficult to check.

  \smallskip
  Let now $x \in \Rspace^d$ be a point so that there are $d+1$ affine functions that map $x$ to $\xi = L_k(x)$, and assume for convenience that their indices are $1, 2, \ldots, d+1$.
  By construction, the $d$-sphere with center $\xxx = (x, \xi)$ that passes through $\zzz$ also passes through $\yyy_i$, for $1 \leq i \leq d+1$.
  Furthermore, a $d$-plane passes above $\xxx$ iff the corresponding point lies inside this $d$-sphere.
  By definition of $L_k$, there are between $k-d$ and $k$ such $d$-planes, but if $x$ is a local minimum, then there are exactly $k$ such $d$-planes.
  Hence, each local minimum gives rise to a $k$-hefty simplex incident to $\zzz$ in $\Rspace^{d+1}$.
  The projection of the $d$-face opposite $\zzz$ to $\Rspace^d$ is the convex hull of the points $2 \lambda_i g_i$, which are positive multiples of the $d+1$ gradients.
  It follows that $0 \in \Rspace^d$ is contained in the interior of this $d$-simplex.
  By construction, $\zzz$ lies vertically above this $d$-face, so the $(d+1)$-dimensional $k$-hefty simplex contains all points sufficiently close and vertically below $\zzz$.
  By Theorem~\ref{thm:covering_for_finite_sets}, the number of such simplices for finitely many points in $\Rspace^{d+1}$ is at most $\binom{d+k}{d}$, which implies the same upper bound for the number of local minima.
\end{proof}
Observe that a local maximum of $L_k$ is a local minimum of $L_{k-d}$, so Proposition~\ref{prop:counting_minima_in_levels} implies that $L_k$ has at most $\binom{k}{d}$ local maxima.
In particular, there are no local maxima unless $k \geq d$.

\subsection{Repeated Minimum Heft}
\label{sec:4.3}

Consider again a collection of affine functions from $\Rspace^d$ to $\Rspace$, in which the $i$-th function is defined by mapping $x$ to $h_i (x) = \scalprod{x}{g_i} + c_i$.
Its zero set, $h_i^{-1} (0)$, is a hyperplane, and we call $h_i^{-1} (-\infty, 0]$ the \emph{negative side} and $h_i^{-1} [0, \infty)$ the \emph{positive side} of the hyperplane.
The hyperplanes decompose $\Rspace^d$ into cells of dimension between $0$ and $d$.
We refer to a $d$-dimensional cell as a \emph{chamber}, and note that it partitions the hyperplanes into those that contain the chamber on their negative and positive sides, respectively.
We call number of hyperplanes of the latter kind the \emph{heft} of the chamber.

\smallskip
We study the following question: ``assuming every chamber has heft at least $k$, how many chambers can have heft exactly $k$?''
There can be at most one chamber with heft $0$, namely the common intersection of the negative sides of the hyperplanes, which is a convex polyhedron.
Beyond $k = 0$, we have an example in which there are $\binom{d+k}{d}$ chambers of minimum heft $k$, but note that this number is independent of the number of hyperplanes.
See Figure~\ref{fig:hexmesh}, which depicts the case of $\binom{2+3}{2} = 10$ chambers of minimum heft $3$ in $\Rspace^2$.
\begin{figure}[hbt]
  \centering \vspace{0.1in}
  \resizebox{!}{2.4in}{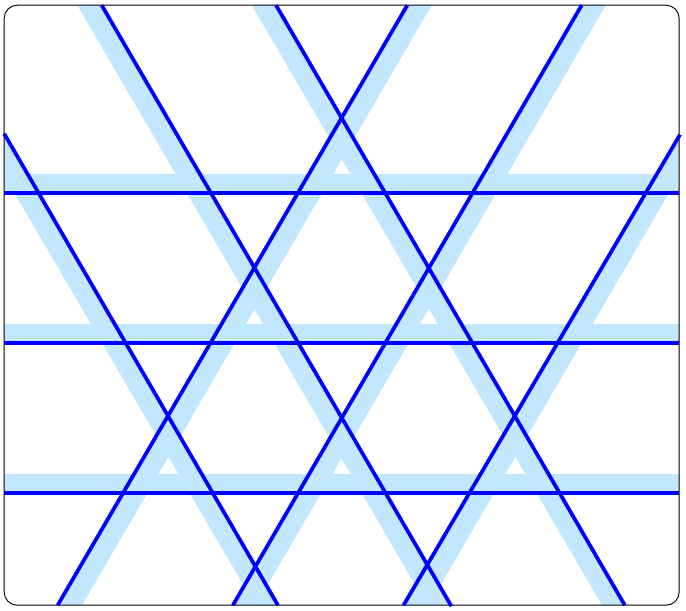}
  \vspace{-0.05in}
  \caption{\footnotesize The lower bound example for $k=3$ in two dimensions.
  The \emph{shading} indicates the positive side of each line.
  The ten chambers with minimum heft $3$ are labeled as such, while all other chambers have heft strictly larger than $3$.}
  \label{fig:hexmesh}
\end{figure}
For the general case, we fix $k \geq 1$ and begin with a regular $d$-simplex, whose size is chosen so that each vertex has distance $2k-1$ from the opposite facet.
For each of the $d+1$ facets, take $k$ parallel hyperplanes as follows: the first contains the facet, and each next hyperplane is at distance $2$ from the previous one and closer to the opposite vertex, which lies on its positive side.
It is not difficult to see that the minimum heft is $k$ and the number of chambers with heft $k$ is $\binom{d+k}{d}$.
We prove that this is indeed the largest possible number of such chambers provided the affine functions are \emph{generic}, by which we mean that for each $1 \leq p \leq d$, the maximum number of hyperplanes that intersect in a common $d-p$ plane is $p$.
In particular, the maximum number of hyperplanes that intersect in a common point is $d$.
\begin{theorem}
  \label{thm:repeated_minimum_heft}
  Consider a generic collection of finitely many affine functions from $\Rspace^d$ to $\Rspace$.
  If every chamber in the decomposition of $\Rspace^d$ defined by the zero sets of these functions has heft at least $k$, then the number of chambers with heft exactly $k$ is at most $\binom{d+k}{d}$.
\end{theorem}
\begin{proof}
  Without loss of generality, we assume that all chambers with minimum heft are bounded. 
  This can be achieved by adding new hyperplanes that carve out bounded pieces from unbounded chambers and do not intersect bounded chambers at all while preserving genericity and the heft numbers of the bounded chambers.
  Next pick weighted points $\aaa_0, \aaa_1, \ldots \in \Rspace^d \times \Rspace$ such that $\sign{\wdist{0}{x} - \wdist{i}{x}} = \sign{h_i(x)}$, for each $i \geq 1$ and every $x \in \Rspace^d$, in which we write $\wdist{j}{x}$ for the power of $x \in \Rspace^d$ from $\aaa_j$.
  To do this, we choose $\aaa_0 = (a_0, w_0)$ arbitrarily, and for each $i \geq 1$, we pick $a_i$ on the line normal to the $i$-th hyperplane that passes through $a_0$, such that $a_i - a_0$ points from the negative to the positive side, and $w_i \in \Rspace$ such that the $i$-th hyperplane is the bisector of $\aaa_0$ and $\aaa_i = (a_i, w_i)$.

  \smallskip
  Let $x$ be a point in the interior of a chamber with minimum heft, which we assume is $k$, and let $\xxx = (x, w)$ be the weighted point such that $\xxx$ and $\aaa_0$ are orthogonal; that is: $\Edist{x}{a_0}^2 = w + w_i$.
  By construction, $k$ of the $\aaa_i$ are closer than orthogonal to $\xxx$, while the others are further than orthogonal from $\xxx$.
  If we move $x$ to a vertex of this chamber, then $d$ of the latter weighted points become orthogonal to $\xxx$ while none of the others change their status.
  Hence, each vertex of the chamber defines a $k$-hefty simplex with vertex $a_0$.

  \smallskip
  When we walk along the boundary of the chamber, we also go around $a_0$ exactly once.
  In other words, these $k$-hefty simplices cover any sufficiently small neighborhood of $a_0$ exactly once.
  To see this, we remove the $k$ hyperplanes that have the chamber on their positive sides, and with them the corresponding weighted points.
  Since the heft of the chamber used to be a global minimum, none of the removed hyperplanes contains a facet of the chamber and, by assumption of genericity, neither a vertex of the chamber.
  So the chamber remains as is, only that it now has heft $0$.
  Similarly, the mentioned $k$-hefty simplices are now $0$-hefty simplices.
  In other words, they belong to the weighted Delaunay mosaic of $\aaa_0, \aaa_1, \ldots$ and, in particular, they form the star of $a_0$ in this mosaic.
  But this mosaic covers $\Rspace^d$ at most once.
  Since the chamber is bounded, $a_0$ is not a boundary vertex of this mosaic, so every sufficiently small neighborhood of $a_0$ is covered exactly once, which implies that the $k$-hefty simplices cover every such neighborhood of $a_0$ exactly once.

  \smallskip
  In conclusion: for each chamber with minimum heft $k$, we get a single layer of $k$-hefty simplices incident to $a_0$.
  By Theorem~\ref{thm:covering_for_weighted_points}, the number of layers is at most $\binom{d+k}{d}$, which implies the claimed upper bound on the number of such chambers.
\end{proof}

\begin{remark}\label{rem:repeated_minimum_heft_sphere}
  A related question considers finitely many points on the unit sphere in $\Rspace^{d+1}$ and asks how many hemispheres contain exactly $k$ of the points, assuming no hemisphere contains fewer than $k$ points; compare with the algorithmic task of finding a hemisphere that contains the minimum/maximum number of points of a given finite set in $\Sspace^d$ \cite{JP78}.
  The dual of this question is indeed similar to the one answered by Theorem~\ref{thm:repeated_minimum_heft}, except that we have hemispheres of $\Sspace^d$ instead of half-spaces in $\Rspace^d$.

  \smallskip
  The upper bound generalizes to the hemisphere setting, but only if there are sufficiently many hemispheres.
  To see this, suppose there are $n$ hemispheres in general position and $\binom{d+k}{d} + N$ chambers with minimum heft $k$.
  Note that every hemisphere has at most $\binom{d+k}{d}$ $k$-hefty chambers it does not cover.
  Indeed, we can centrally project the complement of the hemisphere to $\Rspace^d$, where Theorem~\ref{thm:repeated_minimum_heft} implies this bound.
  Hence, the hemisphere covers at least $N$ $k$-hefty chambers, but each such chamber is covered only $k$ times, which implies
  \begin{align}
    n &\leq k \left[ \binom{d+k}{d} + N \right] / \,N = k + \frac{k}{N} \binom{d+k}{d} .
  \end{align}
  For fixed $d$ and $k$, the right-hand side decreases with increasing $N$.
  Setting $N = 1$ and taking the contrapositive, this says that the upper bound stated in Theorem~\ref{thm:repeated_minimum_heft} for half-spaces in $\Rspace^d$ also applies to hemispheres in $\Sspace^d$ provided there are more than $k + k \binom{d+k}{d}$ hemispheres.

  \smallskip
  For the general case---in which the number of hemispheres is not restricted to exceed the given bound---we get $\binom{d+k}{d} + \binom{d+k-1}{d}$ as an upper bound.
  Indeed, any hemisphere in the set has at most $\binom{d+k}{k}$ $k$-hefty chambers on its negative side and, after removing the hemisphere, there are at most $\binom{d+k-1}{d}$ $(k-1)$-hefty chambers on its positive side.
  For some values of $k$ and $n$ this bound is indeed tight.
  Take for example the regular $(2k+1)$-gon on $\Sspace^1 \subseteq \Rspace^2$, which dualizes to $n = 2k+1$ half-circles with $2k+1$ $k$-hefty chambers,
  while all other chambers have heft $k+1$.
  This number agrees with the upper bound for the case $d=1$.
  
\Skip{
  \smallskip
  For small $n$ compared to $k$, there are in fact examples with more than $\binom{d+k}{d}$ $k$-hefty chambers, e.g.\ the vertices of the regular $(2k+1)$-gon on $\Sspace^1 \subseteq \Rspace^2$, which dualize to $2k+1$ half-circles with $2k+1$ $k$-hefty chambers.
  This is the minimum since all other chambers have heft $k+1$.
  Furthermore, $2k+1 > k+1$, in which the right-hand side is the upper bound if there are more than $k^2 + 2k$ half-circles.
  So we have the seemingly paradoxical situation in which the maximum number of minimum heft chambers is larger for small numbers of hemispheres than for large numbers of hemispheres.

  \AG{Nevertheless, it is possible to obtain a general upper bound for arbitrary $n$ as well. If we pick a hemisphere from our collection, its complement contains at most $\binom{d+k}{d}$ chambers of the minimum heft $k$ as we can see from central projection we mentioned before. 
  Similarly, the central projection of the hemisphere itself gives a collection of affine functions with minimum heft at least $k-1$. 
  This implies that the hemisphere contains at most $\binom{d+k-1}{d}$ chambers of heft $k$.
  Overall, there are at most $\binom{d+k}{d}+\binom{d+k-1}{d}$ chambers with minimum heft $k$. Note that this bound is sharp at least for some $d$ and $k$ as illustrated by our regular $(2k+1)$-gon example for $d=1$ above.}
}
\end{remark}

\subsection{Worpitzky's Identity for Eulerian Numbers}
\label{sec:4.4}

As the fourth and last application of our covering results, we show how the multiplicities from Theorem~\ref{thm:global_covering} are related to the volumes of hypersimplices and to Eulerian numbers.
Specifically, we show that de Laplace's relation for hypersimplices \cite{Lap86} implies Worpitzky's identity for Eulerian numbers \cite{Wor83}, and vice versa.

\smallskip
We begin by introducing the three main concepts we need in this subsection.
Letting $d$ be a positive integer, the number of \emph{descents} in a permutation $j_1, j_2, \ldots, j_d$ of $1, 2, \ldots, d$ is the number of indices $1 \leq i \leq d-1$ such that $j_i > j_{i+1}$.
For $0 \leq k \leq d-1$, the \emph{Eulerian number} for $d$ and $k$ is the number of permutations of $1,2,\ldots,d$ with exactly $k$ descents.
For example, $A(d,0) = A(d,d-1) = 1$, for every $d$, and $\sum_{k=0}^{d-1} A(d,k) = d!$.
Less obvious is Worpitzky's identity \cite{Wor83} between two polynomials from more than a century ago:
\begin{align}
  \sum\nolimits_{k=0}^{d-1} A(d,k) \binom{x+k}{d} &= x^d .
  \label{eqn:Worpitzky_identity}
\end{align}
Next, let $x_0, x_1, \ldots, x_d$ be $d+1$ affinely independent points and write $\Delta = \conv{\{x_0, x_1, \ldots, x_d\}}$ for the $d$-simplex they span.
A \emph{$p$-fold sum} is obtained by selecting $p$ of the $d+1$ points, considering them as vectors, and returning the point that corresponds to the sum of these vectors.
For each $1 \leq p \leq d$, the convex hull of the $p$-fold sums of the $d+1$ points is a $d$-dimensional convex polytope referred to as a \emph{$d$-hypersimplex of order} $p$, denoted $\Delta_d^{(p)}$.
Clearly, $\Delta_d^{(1)} = \Delta$, and more generally, $\Delta_d^{(p)}$ is a homothetic copy of the convex hull of the barycenters of all $(p-1)$-dimensional faces of $\Delta$.
Since the barycenter of a $(p-1)$-simplex is $\sfrac{1}{p}$ times the sum of its vertices, the volume of that polytope is $\sfrac{1}{p^d}$ times the volume of the homothetic hypersimplex.
Define the \emph{relative volume} of $\Delta_d^{(p)}$ as $v(d,p) = {\vol_d(\Delta^{(p)}_d)} / {\vol_d(\Delta)}$, and observe that it does not depend on the choice of $\Delta$.
Again more than a century ago, de Laplace~\cite{Lap86} proved that these relative volumes are Eulerian numbers:
\begin{align}
  v(d,k+1) &= A(d,k) ;
  \label{eqn:Laplace_relation}
\end{align}
see also the combinatorial proof of the same equation by Stanley~\cite{Sta77} and an interpretation of de Laplace's result by Foata \cite{Foa77}.
The third concept is the dual of the order-$n$ Voronoi tessellation of a finite set $A \subseteq \Rspace^d$, introduced in 1990 by Aurenhammer~\cite{Aur90}.
Referring to this dual as the \emph{order-$n$ Delaunay mosaic} of $A$, denoted $\Delaunay{n}{A}$, it is defined by its $d$-cells, each the convex hull of a collection of averages of $n$ points selected from $A$.
Specifically, for each $1 \leq p \leq d$ and every $(n-p)$-hefty simplex, take all sets of cardinality $n$ that contain all $n-p$ points inside the circumscribed sphere together with any $p$ points on the circumscribed sphere.
E.g.\ for $p=1$, we get $d+1$ averages whose convex hull is a homothetic copy of the original $d$-simplex, and for $p=2$, we get a homothetic copy of the convex hull of the midpoints of the edges of the $d$-simplex.
Collecting these polytopes, we get $\Delaunay{n}{A}$.
For $n=1$, we have $\Delaunay{1}{A} = \Delaunay{}{A}$, and more generally $\Delaunay{n}{A}$ has a $d$-cell for every $(n-p)$-hefty simplex in which $p$ varies from $1$ to $d$.
Since the vertices are averages of $n$ points, each $d$-cell in $\Delaunay{n}{A}$ has volume $\sfrac{1}{n^d}$ times the volume of the corresponding hypersimplex.
It is now easy to prove the following relation for the relative volumes of the hypersimplices.
\begin{theorem}
  \label{thm:hypersimplex_identity}
  For integers $d, n \geq 1$,
  the relative volumes of the hypersimplices satisfy
  \begin{align}
    \sum\nolimits_{p=1}^{d} v(d,p) \binom{n+d-p}{n-p} &= n^d ,
    \label{eqn:hypersimplex_identity}
  \end{align}
  in which $\binom{n+d-p}{n-p} = 0$ whenever $n-p < 0$.
\end{theorem}
\begin{proof}
  Let $A$ be any Delone set---and not just a thin Delone set---in $\Rspace^d$, so that every ball whose radius exceeds some given constant contains at least one point of $A$.
  Let $R > 0$ be sufficiently large and consider all $d$-cells in $\Delaunay{n}{A}$ that are contained in $[-R,R]^d$.
  Setting $n' = \max \{ 0, n-d \}$, there are $n-n' \leq d$ different types of $d$-cells to be considered, namely homothetic copies of hypersimplices of orders $1$ to $n-n'$ defined by $(n-p)$-hefty simplices for $1 \leq p \leq n-n'$.
  The total volume of these $d$-cells is $(2R)^d + \bigOh{R^{d-1}}$, since we miss only a constant width neighborhood of each facet of the hypercube.

  \smallskip
  Consider now an $(n-p)$-hefty simplex of $A$.
  Generically, its circumscribed sphere passes through $d+1$ points and encloses $n-p$ of the points in $A$.
  By definition of Delone set, the radius of this circumscribed sphere is bounded from above by a constant times $n-p$.
  Furthermore, the $(n-p)$-hefty simplex contains every $d$-cell in $\Delaunay{n}{A}$ it may determine since the vertices of the latter are averages of the vertices of the $(n-p)$-hefty simplex.
  By Theorem~\ref{thm:global_covering}, the $(n-p)$-hefty simplices that define $d$-cells of $\Delaunay{n}{A}$ inside the hypercube therefore cover most of the hypercube exactly $\binom{n+d-p}{n-p}$ times.
  It follows that the total volume of these $p$-hefty simplices is $\binom{n+d-p}{n-p} (2R)^d + \bigOh{R^{d-1}}$.
  By definition of relative volume, the total volume of the corresponding hypersimplices thus is ${v(d,p)} \binom{n+d-p}{n-p} (2R)^d / n^d + \bigOh{R^{d-1}}$.
  Taking the sum for $1 \leq p \leq n-n'$, dividing by $(2R)^d$, and taking the limit as $R$ goes to infinity, we get the claimed relation.
\end{proof}

To see that Theorem~\ref{thm:hypersimplex_identity} together with de Laplace's relation implies Worpitzky's identity and together with Worpitzky's identity implies de Laplace's relation, we change the summation index in \eqref{eqn:hypersimplex_identity} from $p$ to $k=d-p$ and apply $\binom{n+k}{n-d+k} = \binom{n+k}{d}$ to get
\begin{align}
  \sum\nolimits_{k=0}^{d-1} v(d,d-k) \binom{n+k}{d} &= n^d .
  \label{eqn:hypersimplex_again}
\end{align}
Since this relation holds for every positive integer, $n$, it also holds if we treat $\binom{n+k}{d}$ and $n^d$ as polynomials of degree $d$ in $n$.
Substituting $A(d,k) = A(d,d-k-1)$ for $v(d,d-k)$ using de Laplace's relation \eqref{eqn:Laplace_relation}, we get Worpitzky's identity \eqref{eqn:Worpitzky_identity}.
To see the other direction, we observe that the polynomials given by the binomial coefficients are linearly independent, so there is only one way to write $n^d$ as their linear combination, and it is given by Worpitzky's identity.
Comparing \eqref{eqn:hypersimplex_again} with \eqref{eqn:Worpitzky_identity}, we get $v(d,d-k) = A(d,k) = A(d,d-k-1)$, which is \eqref{eqn:Laplace_relation}.

\section{Concluding Remarks}
\label{sec:5}

Given a finite or locally finite set of points, this paper generalizes Delaunay's concept of \emph{simplices with empty circumscribed spheres} \cite{Del34} to what we call \emph{$k$-hefty simplices}, and proves that they cover space a predictable number of times.
We use this insight to give new proofs of old and new results in discrete geometry, thus providing evidence for the potential of these simplices.
Besides points and weighted points in Euclidean space, we study the spherical setting but omit others, like hyperboloic space, to which our results extend as well.
In conclusion, we highlight additional connections to topics in computational geometry and combinatorics, and formulate  questions we expect can be tackled with the methods studied in this paper.

\medskip \noindent
\textbf{Maximal Feasible Subsystems.}
Let $A \xxx \leq \bbb$ be a system of $n$ linear inequalities in $\Rspace^d$ specifying the negative halfspaces considered in Section~\ref{sec:4.3}.
When $k=0$, the system is consistent and the solution set is a convex polyhedron.
For $k>0$, the system is inconsistent and we need to remove at least $k$ inequalities to make it consistent.

\smallskip
For a given inconsistent system of inequalities, the problem of finding any minimal subset whose removal makes the system consistent is the Maximal Feasible Sybsystem ({\sc Max FS}) problem, which is known to be {\sc NP}-hard \cite{AK95}.
Theorem~\ref{thm:repeated_minimum_heft} implies that if a generic linear system can be made consistent by the removal of $k$ inequalities, but not by the removal of fewer than $k$ inequalities, then there are at most $\binom{d+k}{d}$ options to do that.

\medskip \noindent
\textbf{Chambers in Oriented Matroids.}
Every arrangement of oriented hemispheres gives rise to a realizable oriented matroid with chambers marked by sign vectors, where the $i$-th sign records whether the chamber lies inside or outside the $i$-th hemisphere; see e.g.\ \cite{RGZ17}.
Assuming each chamber of a uniform realizable oriented matroid of rank $d+1$ has at least $k$ plus signs, Remark~\ref{rem:repeated_minimum_heft_sphere} gives an upper bound on the number of chambers with exactly $k$ plus signs.
We note that this bound is different when the number of hemispheres is arbitrary and when this number is sufficiently large, in which case the bound does not depend on it.

\smallskip
We conjecture that the same bounds hold for the case of uniform non-realizable oriented matroids, or, equivalently, for the case of simple arrangements of pseudospheres in $\Sspace^d$.

\medskip \noindent
\textbf{More on $k$-facets and $k$-sets.}
As described in Section~\ref{sec:4.1}, our theorems on $k$-hefty simplices can be used to re-prove classic bounds on $k$-sets and $k$-facets.
However, the authors of this paper have yet been unable to re-prove the tight lower bound on the number of ``at most $k$-sets''.
Motivated by the connection to the rectilinear crossing number of the complete graph, $K_n$, \cite{LVWW04} proves $\sum_{j=0}^k F_j(A) \geq 3 \binom{k+2}{2}$ for every generic set $A$ of $n \geq 2k$ points in $\Rspace^2$.
For $n \geq 3k+3$, we may assume that $0 \in \Rspace^2$ is an interior point of the $k$-hull, so we can invert $A$ through the unit circle centered at $0$ and for each $j$-facet get a $j$-hefty triangle incident to $0$.
By Theorem~\ref{thm:local_covering}, the ``at most $k$-hefty'' triangles incident to $0$ cover any sufficiently small neighborhood of $0$ exactly $\binom{k+2}{2}$ times.
Since every triangle has an angle strictly less than $\pi$ at $0$, we therefore have at least $2 \binom{k+2}{2} + 1$ such triangles.
But the bound in \cite{LVWW04} says there are at least $3 \binom{k+2}{2}$ such triangles.

\smallskip
The above argument generalizes to proving that there are at least $2 \binom{d+k}{d} + 1$ ``at most $k$-hefty'' simplices incident to any one point in a generic set in $\Rspace^d$.
This lower bound appears already in \cite[Theorem~6.3]{Wag03}, along with the conjecture that it can be strengthened to at least $(d+1) \binom{d+k}{k}$ ``at most $k$-hefty'' simplices in $\Rspace^d$.

\medskip \noindent
\textbf{From $k$-hefty Simplices to Manifolds.}
A yet open question motivated by the work in this paper is the global connectivity of the $k$-hefty simplices of a given set.
For example, the $k$-hefty triangles of a finite $k$-balanced set in $\Sspace^2$ cover $\binom{k+2}{2}$ times, but this does not necessarily mean that they form that many spheres each covering $\Sspace^2$ once.

\smallskip
By duplicating vertices and edges as needed, the complex of $k$-hefty triangles and their faces can be turned into an orientable $2$-manifold.
However, this $2$-manifold is not necessarily a universal cover of $\Sspace^2$ as there may be vertex stars that cover the local neighborhood multiple times.
It would be interesting to characterize the manifolds that can arise---in $2$ but also higher dimensions---how they cover, and which local combinatorial properties determine the global connectivity of the manifold.

\section*{Acknowledgments}

The authors would like to thank Boris Aronov and J\'anos Pach for pointing us to result of Clarkson that we reprove in Section~\ref{sec:4.2}. We also would like to thank Jes\'us De Loera for a discussion on Maximal Feasible Subsystems.



\begin{thebibliography}{99}

\bibitem{AlGy86}
{\sc N.\ Alon and E.\ Gy\H{o}ri.}
The number of small semispaces of a finite set of points in the plane.
\emph{J.\ Comb.\ Theory A} {\bf 41} (1986), 154--157.


\bibitem{AK95}
{\sc E.\ Amaldi and V.\ Kann.}
The complexity and approximability of finding maximum feasible subsystems of linear relations.
\emph{Theor.\ Comput.\ Sci.} {\bf 147}:1--2 (1995), 181--210.


\bibitem{ANPS93}
{\sc B.\ Aronov, D.Q.\ Naiman, J.\ Pach and M.\ Sharir.}
An invariant property of balls in arrangements of hyperplanes.
\emph{Discrete Comput.\ Geom.} {\bf 10} (1993), 421--425.

\bibitem{Aur87}
{\sc F.\ Aurenhammer.}
Power diagrams: properties, algorithms and applications.
\emph{SIAM J.\ Comput.} {\bf 16} (1987), 78--96.

\bibitem{Aur90}
{\sc F.\ Aurenhammer.}
A new duality result concerning Voronoi diagrams. 
\emph{Discrete Comput.\ Geom.} {\bf 5} (1990), 243–-254.

\bibitem{BFL90}
{\sc I.\ B\'ar\'any, Z.\ F\"uredi and L.\ Lov\'asz.}
On the number of halving planes. 
\emph{Combinatorica} {\bf 10} (1990), 175--183.

\bibitem{Cla93}
{\sc K.L.\ Clarkson.}
A bound on local minima of arrangements that implies the Upper Bound Theorem.
\emph{Discrete Comput.\ Geom.} {\bf 10} (1993), 427--433.

\bibitem{Del34}
{\sc B.\ Delaunay.}
Sur la sph\`{e}re vide.
\emph{Izv.\ Akad.\ Nauk SSSR, Otdelenie Matematicheskii i Estestvennykh Nauk} {\bf 7} (1934), 793--800.

\bibitem{Ede87}
{\sc H.\ Edelsbrunner.}
\emph{Algorithms in Combinatorial Geometry.}
Springer-Verlag, Heidelberg, Germany, 1987.

\bibitem{EGS25}
{\sc H.\ Edelsbrunner, A.\ Garber and M.\ Saghafian.}
On spheres with $k$ points inside.
In ``Proc.\ 41st Ann.\ Sympos.\ Comput.\ Geom., 2025'', 43.1--43.12.

\bibitem{EdOs21}
{\sc H.\ Edelsbrunner and G.\ Osang.}
The multi-cover persistence of Euclidean balls.
\emph{Discrete Comput.\ Geom.} {\bf 65} (2021), 1296--1313.

\bibitem{Foa77}
{\sc D.\ Foata.}
Distributions Euleriennes et Mahoniennes sur le groupe des permutations.
In \emph{Higher Combinatorics} {\bf 31}, ed.: M.\ Aigner, Reidel Dordrecht, Boston, 1977, 27--49.

\bibitem{JP78}
{\sc D.S.\ Johnson, F.P.\ Preparata.}
The densest hemisphere problem. \emph{Theor.\ Comput.\ Sci.} {\bf 6}:1 (1978), 93--107.


\bibitem{Lap86}
{\sc P.-S.\ de Laplace.} 
\emph{{\OE}uvres compl\`etes de Laplace, vol.\ 7 (1878--1912).} Gauthier-Villars, Paris.

\bibitem{LVWW04}
{\sc L.\ Lov\'asz, K.\ Vesztergombi, U.\ Wagner and E.\ Welzl.} Convex quadrilaterals and $k$-sets.
In \emph{Towards a Theory of Geometric Graphs}, Contemporary Mathematics {\bf 342}, ed.: J.\ Pach, Amer.\ Math.\ Soc., 2004,  139--148.

\bibitem{Mat02}
{\sc J.\ Matou\v{s}ek.}
\emph{Lectures on Discrete Geometry.}
Springer, New York, New York, 2002.

\bibitem{PTT07}
{\sc J.\ Pach, G.\ Tardos and G.\ T\'{o}th.}
Indecomposable coverings.
In \emph{Discrete Geometry, Combinatorics and Graph Theory}, J.\ Akiyama, W.-Y.\ Chen, M.\ Kano, X.\ Li and Q.\ Yu (eds.), LNCS {\bf 4381}, Springer, 2007, 135--148.

\bibitem{RGZ17}
{\sc J.\ Richter-Gebert and G.\ Ziegler.}
Oriented matroids.
In \emph{Handbook of Discrete and Computational Geometry}, 3rd edition, eds.: J.E.\ Goodman, J.\ O'Rourke and C.D.\ T\'oth, CRC Press LLC, 2017, 159--184.

\bibitem{Rog64}
{\sc C.A.\ Rogers.}
\emph{Packing and Covering.}
Cambridge Univ.\ Press, Cambridge, England, 1964.

\bibitem{Rub24}
{\sc N.\ Rubin.}
Improved bounds for point selections and halving hyperplanes in higher dimensions.
In ``Proc.\ Ann.\ ACM-SIAM Sympos.\ Discrete Algorithms, 2024'', 4464--4501.

\bibitem{Sta77}
{\sc R.P.\ Stanley.}
Eulerian partitions of a unit hypercube.
In \emph{Higher Combinatorics} {\bf 31}, (M.\ Aigner, ed.), Reidel Dordrecht, Boston, 1977, page~49.

\bibitem{Wag03}
{\sc U.\ Wagner.}
\emph{On $k$-Sets and Their Applications.}
Ph.D Thesis, ETH Z\"urich, Switzerland, 2003.   

\bibitem{Wag08}
{\sc U.\ Wagner.}
$k$-sets and $k$-facets.
In \emph{Discrete and Computational Geometry---20 Years Later}, Contemporary Mathematics {\bf 453}, eds.: J.E.\ Goodman, J.\ Pach and R.\ Pollack, Amer.\ Math.\ Soc., 2008, 443--514.

\bibitem{Wel01}
{\sc E.\ Welzl.}
Entering and leaving $j$-facets.
\emph{Discrete Comput.\ Geom.} {\bf 25} (2001), 351--364.

\bibitem{Wor83}
{\sc J.\ Worpitzky.}
Studien über die Bernoullischen und Eulerischen Zahlen.
\emph{J.\ Reine Angew.\ Math.} {\bf 94} (1883), 203--232.

\bibitem{Zie95}
{\sc G.M.\ Ziegler.}
\emph{Lectures on Polytopes.}
Springer-Verlag, Berlin, Germany, 1995.

\end{thebibliography}
\end{document}